\documentclass[leqno]{article}
\usepackage{amsmath,amssymb,amsthm,amscd,dsfont}
\numberwithin{equation}{section} \allowdisplaybreaks
\begin{document}
\newtheorem{theorem}{Theorem}[section]
\newtheorem{defin}{Definition}[section]
\newtheorem{prop}{Proposition}[section]
\newtheorem{corol}{Corollary}[section]
\newtheorem{lemma}{Lemma}[section]
\newtheorem{rem}{Remark}[section]
\newtheorem{example}{Example}[section]
\title{Characteristic Classes of Lie Algebroid Morphisms}
\author{{\small by}\vspace{2mm}\\Izu Vaisman}
\date{}
\maketitle
{\def\thefootnote{*}\footnotetext[1]%
{{\it 2000 Mathematics Subject Classification: 53D17, 57R20}.
}}
\begin{center} \begin{minipage}{12cm}
A{\footnotesize BSTRACT. We extend R. Fernandes' construction of
secondary characteristic classes of a Lie algebroid to the case of
a base-preserving morphism between two Lie algebroids. Like in the
case of a Lie algebroid, the simplest characteristic class of our
construction coincides with the modular class of the morphism.}
\end{minipage}
\end{center} \vspace*{5mm}
\noindent In \cite{F} R. Fernandes has constructed a sequence of
secondary characteristic classes of a Lie algebroid whose first
element coincides with the modular class. In this note we extend
Fernandes' construction and use the general definition of D. Lehmann
\cite{L} in order to produce secondary characteristic classes of a
base-preserving morphism of two Lie algebroids. In particular, like
in \cite{F}, we get a sequence of secondary characteristic classes
whose first element coincides with the modular class of the morphism
\cite{{GMM},{KW}}. We assume that the reader is familiar with Lie
algebroids and Lie-algebroid connections and will consult
\cite{{F},{Mk},{L},{V}} whenever needed. The framework of the paper
is the $C^\infty$-category. We mention that other constructions of
secondary characteristic classes of Lie algebroids may also be
found in the literature e.g., \cite{{Cr},{CF},{Kb}}. (The author
is grateful to Rui Fernandes for bringing to his attention
reference
\cite{CF}, which led to Proposition \ref{propCF} of the last
section of the paper.)
\section{Selected topics on $A$-Connections}
Let $(A,\sharp_A,[\,,\,]_A)$ be a Lie algebroid and $V$ a vector
bundle with the same base manifold $M^m$ ($m=dim\,M$). By an
$A$-connection we shall understand an $A$-covariant derivative
$\nabla:\Gamma A\times\Gamma V\rightarrow\Gamma V$ ($\Gamma$
denotes the space of cross sections of a vector bundle), written
as $(a,v)\mapsto\nabla_av$, which is $\mathds{R}$-bilinear and has
the properties
\begin{equation}\label{propconex}
\nabla_{fa}v=f\nabla_av,\;\nabla_a(fv)=f\nabla_av+\sharp_Aa(f)v
\hspace{2mm}(f\in C^\infty(M)).\end{equation}

Accordingly, the value $\nabla_av(x)$ depends only on $a(x)$ and
on $v|_{U_x}$ where $U_x$ is a neighborhood of $x\in M$. In order
to write down the local expression of $\nabla$, we take a local
basis $(b_i)_{i=1}^s$ $(s=rank\,A)$ of $\Gamma A$, with the dual
basis $(b^{*i})$ of $\Gamma A^*$ and a local basis $(w_u)_{u=1}^r$
$(r=rank\,V)$ of $\Gamma V$. Then, with the notation
$\Omega^k(A)=\Gamma\wedge^kA^*$ for the space of $A$-{\it forms}
of degree $k$ and using the Einstein summation convention, we get
\begin{equation}\label{conexloc} \nabla_{b_i}w_u=\omega_u^t(b_i)w_t,\;
\omega_u^t=\gamma_{iu}^tb^{*i}\in\Omega^1(A),\end{equation} and we
say that $(\omega_u^t)$ is the {\it local connection matrix}.
Correspondingly, the curvature
$$R_\nabla(a_1,a_2)w=\nabla_{a_1}\nabla_{a_2}w-\nabla_{a_2}\nabla_{a_1}w-
\nabla_{[a_1,a_2]_A}w$$
gets the local expression
\begin{equation}\label{curbloc}
R_\nabla(b_i,b_j)w_u=\Omega_u^t(b_i,b_j)w_t,\;
\Omega_u^t=d_A\omega_u^t-\omega_u^s\wedge\omega_s^t\in\Omega^2(A),\end{equation}
where $d_A$ denotes the $A$-exterior differential \cite{Mk}. We
will say that $(\Omega_u^t)$ is the {\it local curvature matrix}
and a change of the basis $(w_u)$ implies an
$ad(Gl(r,\mathds{R}))$-transformation of $(\Omega_u^t)$. Like in
classical differential geometry, one has the covariant derivative
machinery of $V$-tensors and tensor valued $A$-forms and the
computation of the $d_A\Omega_u^t$ produces the Bianchi identity
that may be written under the form $\nabla\Omega_u^t=0$.

In the study of characteristic classes we shall need the direct
product of two Lie algebroids $p_c:A_c\rightarrow M_c$ $(c=1,2)$
and we recall its definition given in \cite{Mk}. Consider the
pullback bundles $\pi_c^{-1}A_c$, where $\pi_c$ is the projection
of $M_1\times M_2$ on $M_c$. Identify
\begin{equation}\label{pullsections}
\Gamma(\pi_c^{-1}A_c)\equiv\{\sigma:M_1\times M_2\rightarrow
A_c\,/\,p_c\circ\sigma=\pi_c\}\end{equation} and notice that local
bases $(b_i^{(c)})$ of $\Gamma A_c$ have natural lifts to local
bases of $\Gamma(\pi_c^{-1}A_c)$, which will also be denoted by
$(b_i^{(c)})$. Take local cross sections
$$\sigma_{(c)}=\sigma_{(c)}^ib^{(c)}_i,\,
\kappa_{(c)}=\kappa_{(c)}^ib^{(c)}_i\hspace{2mm}(\sigma_{(c)}^i,
\kappa_{(c)}^i\in C^\infty(M_1\times M_2))$$ (there is no summation
on indices in parentheses) and define the following anchors and
brackets
$$\sharp_{(c)}\,\sigma_{(c)}=\sigma_{(c)}^i\sharp_{A_c}
b^{(c)}_i,[\sigma_{(c)},\kappa_{(c)}]_{(c)}=
\sigma_{(c)}^i\kappa_{(c)}^j[b^{(c)}_i,b^{(c)}_j]_{A_c}$$
$$+\{\sigma_{(c)}^j[(\iota_{*c}\sharp_{A_c}b^{(c)}_j)\kappa_{(c)}^i]
-\kappa_{(c)}^j[(\iota_{*c}\sharp_{A_c}b^{(c)}_j)\sigma_{(c)}^i]\}b^{(c)}_i,$$
where $\iota_{*c}$ is the natural injection of $TM_c$ in
$T(M_1\times M_2)$. In these operations, the $M_{(c-1\,{\rm
mod.}\,2)}$-variable is just a passive parameter and, since the
anchor and bracket of each $A_c$ are invariant, the results are
independent of the choice of the bases. Thus, the vector bundles
$\pi_c^{-1}A_c$ are Lie algebroids over $M_1\times M_2$ and the
direct product of the Lie algebroids $A_c$ is the Whitney sum
$\mathcal{A}=\pi_1^{-1}A_1\oplus
\pi_2^{-1}A_2$ endowed with the direct sum of the anchors and brackets of
the two pullbacks (in particular, $[b_i^{(1)},b_j^{(2)}]=0$).
\begin{prop}\label{conectlinks} Let $q:V\rightarrow M_1$ be a vector
bundle on $M_1$. Then, any $\mathcal{A}$-connection
$\tilde{\nabla}$ on the pullback $\pi_1^{-1}(V)$ defines a
differentiable family $\nabla^{(x_2)}$ $(x_2\in M_2)$ of
$A_1$-connections on $V$. Conversely, any $x_2$-parameterized,
differentiable family of $A_1$-connections on $V$ is induced by an
$\mathcal{A}$-connection on $\pi_1^{-1}(V)$.\end{prop}
\begin{proof} Assume that we have the
covariant derivatives $\tilde{\nabla}_{\sigma^i_{(1)}b_i^{(1)}+
\sigma^j_{(2)}b_j^{(2)}}(\nu^uw_u)$, where $\sigma^i_{(c)},\nu^u$
are local, differentiable functions on $M_1\times M_2$. Then, the
required family of connections on $V$ is given by the covariant
derivatives
$$(\nabla^{(x_2)}_{\xi^ib_i^{(1)}}(\eta^uw_u))(x_1)=
(\tilde{\nabla}_{\xi^ib_i^{(1)}}(\eta^uw_u))(x_1,x_2),$$ where
$\xi^i,\eta^u$ are local, differentiable functions on $M_1$,
$x_1\in M_1,x_2\in M_2$ and we use an identification like
(\ref{pullsections}) for $V$. Notice that, if the local connection
matrices of $\tilde{\nabla}$ are
\begin{equation}\label{eqfamilie}
\tilde\omega_u^v=\gamma_{(1)iu}^v(x_1,x_2)b_{(1)}^{*i}
+\gamma_{(2)ju}^v(x_1,x_2)b_{(2)}^{*j},\end{equation} the
connection $\nabla^{(x_2)}$ has the matrices
$\gamma_{(1)iu}^v(x_1,x_2)b_{(1)}^{*i}$ with the fixed value of
$x_2$. Conversely, if the family $\nabla^{(x_2)}$ is given, we get
an $\mathcal{A}$-connection $\tilde{\nabla}$ by adding the local
equations $\tilde{\nabla}_{b_j^{(2)}}w_u=0$. The local matrices of
this connection $\tilde{\nabla}$ are the same as the matrices of
$\nabla^{(x_2)}$ where $x_2$ is allowed to vary in $M_2$.
\end{proof}

In particular, we may apply Proposition \ref{conectlinks} for
$M_1=M,M_2=I=\{0\leq \tau\leq1\},A_1=A,A_2=TI$. Then, an $
\mathcal{A}$-connection $\tilde{\nabla}$ on $\pi_1^{-1}(V)$ is called a
{\it link} between the $A$-connections $\nabla^0,\nabla^1$ on $V$.
Formula (\ref{eqfamilie}) shows that the local connection forms of
$\tilde{\nabla}$ are given by
\begin{equation}\label{eqlink}
\tilde{\omega}_u^v=\omega_{(\tau)u}^v+\lambda_u^v(x,\tau)d\tau,\end{equation}
where $\omega_{(\tau)}$ is the local connection matrix at the
fixed value $\tau$ and $\lambda_u^v\in C^\infty(M\times I)$. A
simple calculation gives the corresponding local curvature forms
\begin{equation}\label{curblink}
\tilde{\Omega}_u^v=\Omega_{(\tau)u}^v+\Lambda_u^v\wedge d\tau,
\end{equation} where $$
\Lambda_u^v=d_A\lambda_u^v+
\lambda_u^w\omega_{(\tau)w}^v-\lambda_w^v\omega_{(\tau)u}^w+
\frac{\partial\omega_{(\tau)u}^v}{\partial \tau}$$
(the partial derivative with respect to $\tau$ is applied to the
coefficients of the form).

Now, we present another ``selected topic". Let $V\rightarrow M$ be
a vector bundle of rank $r$ endowed with either a positive,
symmetric tensor $g_+\in\Gamma\odot^2V^*$ or a $2$-form
$g_-\in\Gamma\wedge^2V^*$. We shall say that $(V,g_{\pm})$ is a
{\it quasi-(skew)-metric vector bundle}. Notice that we do not ask
$rank\,g_\pm$ to be constant on $M$. An $A$-connection $\nabla$ on
$V$ such that $\nabla g_\pm=0$ will be called a {\it
quasi-(skew)-metric connection}. If $g_{\pm}$ is non degenerate,
the particle ``quasi" will not be used and the connection is
called {\it orthogonal} for $g_+$ and {\it symplectic} for $g_-$.
For a Lie algebroid $A$ over $M$ we shall denote by $L$ a generic,
integral leaf of the distribution $im\,\sharp_A$ and by $L_x$ the
leaf through the point $x\in M$. In what follows we establish
properties of a quasi-(skew)-metric connection that are relevant
to the construction of characteristic classes.
\begin{prop}\label{theorem1} Assume that there exists a
quasi-(skew)-metric connection $\nabla$ on $(V,g_\pm)$. Then, the
following properties hold. 1. If $x\in M$ and $k\in\Gamma V$ is
such that $k|_{L(x)}\in K|_{L(x)}$ $(K=ann\,g_\pm)$, then
$\nabla_ak(x)\in K_x$, $\forall a\in\Gamma A$. 2.
$q=rank\,g_\pm|_L$ is constant along each leaf $L$ and $\forall
x\in M$ there exists an open neighborhood $U_x$ where $V$ has a
local basis of cross sections of the form $(s_h,t_l)$
$(h=1,...,q,\,l=1,...,r-q)$ such that $t_l|_{U_x\cap L_x}\in
K|_{U_x\cap L_x}$ and the projections $[s_h]=s_h\,({\rm mod.}\,K)$
define a canonical basis of the (skew)-metric vector bundle
$((V/K)|_{U_x\cap L_x},g'_\pm)$, where $g'_\pm$ is non-degenerate
and induced by $g_\pm$. 3. With respect to this basis, the
$A$-connection $\nabla$ has local equations
\begin{equation}\label{eclocaleconex} \nabla s_h=\varpi_{(1)h}^ks_k
+\varpi_{(2)h}^pt_p,\; \nabla
t_l=\varpi_{(3)l}^ks_k+\varpi_{(4)l}^pt_p,\end{equation} where the
coefficients are local $1$-$A$-forms, $\varpi_{(3)l}^k(x)=0$ and
$(\varpi_{(1)h}^k(x))\in o(q)$, the orthogonal Lie algebra, in the
$g_+$-case, $(\varpi_{(1)h}^k(x))\in sp(q,\mathds{R})$, the
symplectic Lie algebra, in the $g_-$-case. 4. The curvature of
$\nabla$ has the local expression
\begin{equation}\label{eclocalecurb} R_\nabla s_h=\Phi_{(1)h}^ks_k
+\Phi_{(2)h}^pt_p,\; R_\nabla
t_l=\Phi_{(3)l}^ks_k+\Phi_{(4)l}^pt_p,\end{equation} where the
coefficients are local $2$-$A$-forms and $\Phi_{(3)l}^k(x)=0$,
$(\Phi_{(1)h}^k(x))\in o(q)$ in the $g_+$-case,
$(\Phi_{(1)h}^k(x))\in sp(q,\mathds{R})$ in the
$g_-$-case.\end{prop}
\begin{proof} 1. For any $a\in\Gamma A,v\in\Gamma V$ one has
\begin{equation}\label{conexsympl}
(\nabla_ag_\pm)_x(v(x),k(x))=(\sharp_Aa)_x(g_\pm(v,k))-
g_{\pm,x}(\nabla_av(x),k(x))\end{equation}
$$-g_{\pm,x}(v(x),\nabla_ak(x))=0.$$ Since
$(\sharp_Aa)_x(g_\pm(v,k))$ depends only on $k|_{L(x)}\in
ann\,g_\pm$, it vanishes, and we get the required result.

2. $\nabla g_\pm=0$ is equivalent with the fact that $g_\pm$ is
preserved by parallel translations along paths in a leaf $L$
\cite{F}, therefore, $g_{\pm}$ has a constant rank $q$ along
$L_x$. This implies the existence of bases with the required
properties on a neighborhood $U_x\cap L$ of $x$. (In the metric
case canonical means orthonormal and in the skew-metric case
canonical means symplectic.) Then, take any extension of such a
basis to $U_x$ and shrink the neighborhood $U_x$ as needed to
ensure the linear independence of the extended cross sections.

3. The equality $\varpi_{(3)l}^k(x)=0$ is an immediate consequence
of part 1. Then, in (\ref{conexsympl}), replace $v,k$ by
$s_h,s_k$. Since the canonical character of the basis $(s_{h}|_L)$
implies $g_{\pm}(s_{h}|_L,s_{k}|_L)= const.$, we get
$$g_{\pm,x}(\nabla_as_h(x),s_k(x))
+g_{\pm,x}(s_h(x),\nabla_as_k(x))=0,$$ whence,
$(\varpi_{(1)i}^j(x))\in o(q),\, sp(q,\mathds{R})$, respectively.

4. The (skew)-metric condition (\ref{conexsympl}) also implies
$$\sharp_A
[a_1,a_2]_A(g_\pm(v_1,v_2))=g_\pm(\nabla_{[a_1,a_2]_A}v_1,v_2)+
g_\pm(v_1,\nabla_{[a_1,a_2]_A}v_2),$$ where $a_1,a_2\in\Gamma A,
v_1,v_2\in\Gamma V$, whence, after some obvious cancellations we
get
\begin{equation}\label{auxcurv}\sharp_A[a_1,a_2]_A(\omega(v_1,v_2)) =
-\omega(R_\nabla(a_1,a_2)v_1,v_2)-\omega(v_1,R_\nabla(a_1,a_2)v_2).\end{equation}
Like in the proof of 3, (\ref{auxcurv}) for $s_h,s_k$ implies
$(\Phi_{(1)i}^j(x))\in o(q),\, sp(q,\mathds{R})$, respectively.
Then, (\ref{auxcurv}) for $s_h,t_l$ together with part 1 of the
proposition implies $\Phi_{(3)l}^k(x)=0$.
\end{proof}

In the theory of characteristic classes we need the Weil algebra
$I(Gl(r,\mathds{R}))$ $=\oplus_{k\geq0}I^k(Gl(r,\mathds{R}))$,
where $I^k(Gl(r,\mathds{R}))$ is the space of real, ad-invariant,
symmetric, $k$-multilinear functions (equivalently, invariant,
homogeneous polynomials of degree $k$) on the Lie algebra of the
general, linear group $(r=rank\,V)$. Using the exterior product
$\wedge$, such functions may be evaluated on arguments that are
local matrices of $\wedge$-commuting $A$-forms on $M$ with
transition functions of the adjoint type and the result is a
global $A$-form on $M$ (e.g.,
\cite{V}). Secondary characteristic classes appear as a
consequence of vanishing phenomena encountered in the evaluation
process described above. We shall need the following vanishing
phenomenon (see \cite{F}):
\begin{prop}\label{theorem2} If the bundle $V$ endowed with the
form $g_\pm$ has a connection $\nabla$ such that $\nabla g_\pm=0$
and if $R_\nabla(a_1,a_2)k_x=0$ for $x\in M$, $a_1,a_2\in\Gamma
A$, $k\in ker\,g_{\pm,x}$, then, $\forall\phi\in
I^{2k-1}(Gl(r,\mathds{R}))$, one has $\phi(\Phi)=0$, where $\Phi$
is the local curvature matrix of the connection $\nabla$.
\end{prop}
\begin{proof} By $\phi(\Phi)$ we understand the
evaluation of $\phi$ where all the arguments are equal to $\Phi$.
It is known that (with a harmless abuse of terminology and
notation) the required functions $\phi$ are spanned by the Chern
polynomials
\begin{equation}\label{Chern}
c_h(F)=\frac{1}{h!}\delta^{v_1...v_h}_{u_1...u_h}f_{v_1}^{u_1}
...f_{v_h}^{u_h}\end{equation} ($\delta^{...}_{...}$ is the
multi-Kronecker index), which are the sums of the principal minors
of order $h$ in $det(F-\lambda Id)$ $(F\in gl(r,\mathds{R}))$.
With the notation of Proposition \ref{theorem1} and since
$R_\nabla(a_1,a_2)k_x=0$, we have to take
$$F=\left(\begin{array}{cc}
\Phi_{(1)}&0\vspace{2mm}\\\Phi_{(2)}&0\end{array}\right).$$ Therefore,
$\forall x\in M$, we have $c_h(F)=c_h(\Phi_{(1)x})$. It is known
that the polynomials $c_{2l-1}$ vanish on $o(q)$ and on
$sp(q,\mathds{R})$ (in the first case $\Phi_{(1)}$ is
skew-symmetric; for the second case see Remark 2.1.10 in \cite{V},
for instance).\end{proof}
\section{Secondary characteristic classes}
A brief exposition of the classical theory of real characteristic
classes may be found in \cite{V}. In this section, we present a
Lie algebroid version of the basic facts of the theory.

Consider the direct product Lie algebroid $\mathcal{A}=A\times
T\Delta^k\rightarrow M\times\Delta^k$, where
$$\Delta^k=\{(t_0,t_1,...,t_k)\in\mathds{R}^{k+1}\,/\,t_h\geq0,\,
\sum_{h=0}^k t_h=1\}$$ is the standard $k$-simplex, $A$ is a Lie
algebroid over $M$ and $T\Delta^k$ is the tangent bundle of
$\Delta^k$ endowed with the standard orientation
$\kappa=dt^1\wedge...\wedge dt^k$. Then,
$\forall\Phi\in\Omega^*(\mathcal{A})$, the fiber-integral
$\int_{\Delta^k}\Phi$ is defined as zero except for the case
$$\Phi=\alpha\wedge\kappa,\hspace{3mm}\alpha=
\frac{1}{p!}\alpha_{i_1...i_p}(x,t)b^{*i_1}\wedge...\wedge b^{*i_p}
\;\;(x\in M,t\in\Delta^k)$$ when
$$\int_{\Delta^k}\Phi=\frac{1}{p!}
\left(\int_{\Delta^k}\alpha_{i_1...i_p}(x,t)\kappa\right)
b^{*i_1}\wedge...\wedge b^{*i_p}\in\Omega^p(A)$$ ($b_i$ is a local
basis of cross sections of $A$). The same proof as in the
classical case (e.g., \cite{V}, Theorem 4.1.6) yields the Stokes
formula:
\begin{equation}\label{Stokes}\int_{\Delta^k}d_\mathcal{A}
\Phi-d_{A}\int_{\Delta^k}\Phi=
(-1)^{deg\,\Phi-k}\int_{\partial{\Delta^k}}\iota^*\Phi,\hspace{5mm}\iota:
\partial{\Delta^k}\subseteq{\Delta^k}.\end{equation}

Assume that we have $k+1$ $A$-connections $\nabla^{(s)}$ on the
vector bundle $V\rightarrow M$ that have the local connection
matrices $\omega_{(\alpha)}$ $(\alpha=0,...,k)$ with respect to
the local basis $(w_u)$ of $V$. Then, the convex combination
\begin{equation}\label{averagecon}
\nabla^{(t)}=\sum_{\alpha=0}^kt^\alpha\nabla^{\alpha},\hspace{3mm}
t=(t^0,...,t^k)\in\Delta^k,\end{equation} defines a family of
$A$-connections parameterized by $\Delta^k$ with the corresponding
$\mathcal{A}$-connection $\tilde{\nabla}$ on
$\pi_1^{-1}(V)\rightarrow M\times\Delta^k$
$(\pi_1:M\times\Delta^k\rightarrow M)$. The connection and
curvature matrices of $\tilde{\nabla}$ will be denoted by
$\tilde{\omega},\tilde{\Omega}$; generally, the curvature matrix
of a connection will be denoted by the upper case of the letter
that denotes the connection matrix. There exists a homomorphism
$$\Delta(\nabla^0,...,\nabla^k):I^h(Gl(r,\mathds{R}))
\rightarrow\Omega^{2h-k}(A),$$ defined by R. Bott in the classical
case, given by
\begin{equation}\label{Bottmare}
\Delta(\nabla^0,...,\nabla^k)\phi=(-1)^{\left[\frac{k+1}{2}\right]}
\int_{\Delta^k}\phi(\tilde{\Omega}),\hspace{2mm}\phi\in
I^h(Gl(r,\mathds{R})).\end{equation} Moreover, Bott's proof in the
classical case (\cite{V}, Proposition 4.2.3) also holds in the Lie
algebroid version and yields the following formula
\begin{equation}\label{diffBott}
d_A(\Delta(\nabla^0,...,\nabla^k)\phi) =
\sum_{\alpha=0}^k(-1)^\alpha\Delta(\nabla^0,...,\nabla^{\alpha-1},
\nabla^{\alpha+1},...,\nabla^k)\phi.\end{equation}

Let $\nabla$ be an $A$-connection on the  vector bundle
$V\rightarrow M$. As a consequence of the Bianchi identity,
$\forall\phi\in I^h(Gl(r,\mathds{R}))$,
$\Delta(\nabla)\phi\in\Omega^{2h}(A)$ is a $d_A$-closed $A$-form
and the $A$-cohomology classes defined by the $A$-forms
$\Delta(\nabla)\phi$ are called the $A$-{\it principal
characteristic classes of $V$}
\cite{F}. If $\nabla^0,\nabla^1$ are two $A$-connections, formula
(\ref{diffBott}) yields
\begin{equation}\label{Bott2conex}\Delta(\nabla^1)\phi-\Delta(\nabla^0)\phi=
d_A\Delta(\nabla^0,\nabla^1)\phi.\end{equation} Therefore, the
principal characteristic classes do not depend on the choice of
the connection.

The $\mathcal{A}$-connection $\tilde{\nabla}$ to be used in
definition (\ref{Bottmare}) of $\Delta(\nabla^0,\nabla^1)\phi$ is
the link between $\nabla^0,\nabla^1$ given by the family of
$A$-connections
$$\nabla^{(\tau)}=(1-\tau)\nabla^0+\tau\nabla^1=\nabla^0+\tau D,\hspace{3mm}
D=\nabla^1-\nabla^0,\;\;(\tau\in I).$$ For this link, we have
(\ref{eqlink}) and (\ref{curblink}) where
$\lambda_u^v=0,\partial\omega_{(\tau)u}^v/\partial\tau=\alpha$,
the local matrix of the connection difference $D$, and formula
(\ref{Bottmare}) yields
\begin{equation}\label{defDelta} \Delta(\nabla^0,\nabla^1)\phi
=h\int_0^1\phi(\alpha,\underbrace{\Omega_{(\tau)},...,\Omega_{(\tau)}}_{(h-1)-{\rm
times}})d\tau,\end{equation} where
\begin{equation}\label{Omegat}
\Omega_{(\tau)}=(1-\tau)\Omega_{(0)}+\tau\Omega_{(1)}
+\tau(1-\tau)\alpha\wedge\alpha\end{equation} is the local
curvature matrix of the connection $\nabla^{(\tau)}$.

We shall use the Lehmann version of the theory of secondary
characteristic classes \cite{{L},{V}}. Let $(J_0,J_1)$ be two
(proper) homogeneous ideals of $I=I(Gl(r,\mathds{R}))$. Define the
algebra
\begin{equation}\label{Weil}W(J_0,J_1)=(I/J_0)\otimes(I/J_1)\otimes(\wedge(I^+))
\hspace{3mm}(I^+=\oplus_{k>0}I^k),\end{equation} with the graduation
$$deg\,[\phi]_{J_0}
=deg\,[\phi]_{J_1}=2h,\,deg\,\hat\phi=2h-1,$$ and the differential
$$d[\phi]_{J_0}=d[\phi]_{J_1}=0,\,d\hat\phi=[\phi]_{J_1}-[\phi]_{J_0},$$
where we refer to the three elements defined by $\phi\in I^h$ in
the factors of $W$.

Now, take a vector bundle $V\rightarrow M$ and two $A$-connections
$\nabla^0,\nabla^1$ on $V$ such that $J_c\subseteq
ker\,\Delta(\nabla^c)$, $c=0,1$. By putting
\begin{equation}\label{defrho}\rho[\phi]_{J_0}=\Delta(\nabla^0)\phi,\,
\rho[\phi]_{J_1}=\Delta(\nabla^1)\phi,\,
\rho\hat\phi=\Delta(\nabla^0,\nabla^1)\phi,\end{equation} we get a
homomorphism of differential graded algebras
$$\rho(\nabla_0,\nabla_1):W(J_0,J_1)\rightarrow\Omega(A)$$
with an induced cohomology homomorphism
$$\rho^{*}(\nabla^0,\nabla^1):H^*(W(J_0,J_1))\rightarrow
H^*(A).$$ The cohomology classes in $im\,\rho^{*}$ that are not
principal characteristic classes are called $A$-{\it secondary
characteristic classes}.

If $J$ is a homogeneous ideal of $I(Gl(r,\mathds{R}))$, two
$A$-connections $\nabla,\nabla'$ on $V$ are called $J$-{\it
homotopic connections} if there exists a finite chain of links
$\tilde{\nabla}^0,...,\tilde{\nabla}^n$ that starts with $\nabla$,
ends with $\nabla'$ and is such that $J\subseteq\cap_{l=0}^n
ker\,\Delta(\tilde{\nabla}^l)$. By replacing the usual Stokes'
formula by formula (\ref{Stokes}) in the proof of Theorem 4.2.28
of \cite{V}, one gets
\begin{prop}\label{invhomot} {\rm \cite{L}} The cohomology homomorphism
$\rho^{*}(\nabla^0,\nabla^1)$ remains unchanged if
$\nabla^0,\nabla^1$ are replaced by $J_0,J_1$-homotopic
connections $\nabla^{'0},\nabla^{'1}$, respectively ($J_c\subseteq
ker\,\Delta(\nabla^c)$, $J_c\subseteq ker\,\Delta(\nabla^{'c})$,
$c=0,1$).\end{prop}
\begin{corol}\label{corolhomotopie} The secondary characteristic
classes are invariant by any $J_0J_1$-homotopy of the
connections.\end{corol}

Denote by $J_{\rm odd}\subseteq I(Gl(r,\mathds{R}))$ the ideal
spanned by $\{\phi\in I^{2h-1}(Gl(r,\mathds{R})),\\h=1,2,...\}$.
As explained in Proposition \ref{theorem2}, if $\nabla$ is an
orthogonal connection for some metric $g$ on the vector bundle
$V$, then $J_{\rm odd}\subseteq ker\,\Delta(\nabla)$. Notice that
there always exist positive definite metrics $g$ on $V$ and
corresponding metric $A$-connections $\nabla$, $\nabla g=0$ (e.g.,
take $\nabla_a=\nabla'_{\sharp_Aa}$, where $\nabla'$ is a usual
orthogonal connection on $(V,g)$). Furthermore, any two orthogonal
$A$-connections on $V$ are $J_{\rm odd}$-homotopic. Indeed, if
$\nabla,\nabla'$ are orthogonal for the same metric $g$, then
$(1-\tau)\nabla+\tau\nabla'$ $(0\leq\tau\leq1)$ defines an
orthogonal link. If orthogonality is with respect to different
metrics $g,g'$, then $(1-\tau)g+\tau g'$ is a metric on the
pullback of $V$ to $M\times[0,1]$ and a corresponding metric
connection provides an orthogonal link between two orthogonal
connections $\bar{\nabla},\bar{\nabla}'$ with the metrics $g,g'$,
respectively. Thus, there exists a chain of three orthogonal links
leading from $\nabla$ to $\bar{\nabla}$, from $\bar{\nabla}$ to
$\bar{\nabla}'$ and from $\bar{\nabla}'$ to $\nabla'$, which
proves the $J_{\rm odd}$-homotopy of $\nabla,\nabla'$.

Now, let $(V,g_\pm)$ be a quasi-(skew)-metric vector bundle that
has a $K$-flat quasi-(skew)-metric connection $\nabla^1$
$(K=ann\,g_\pm)$. Then, Proposition \ref{theorem2} tells us that
$J_{\rm odd}\subseteq ker\,\Delta(\nabla^1)$. Accordingly (like in
the case of the Maslov classes \cite{V}), if we also take an
orthogonal $A$-connection $\nabla^0$ on $V$, we shall obtain
secondary characteristic classes corresponding to the ideals
$J_0=J_1=J_{\rm odd}$.

Following \cite{V}, Theorem 4.2.26, we may replace the algebra
$W(J_0,J_1)$ by the algebra
\begin{equation}\label{Maslov} \tilde{W}=
\mathds{R}[c_2,c_4,...]\otimes\mathds{R}[c'_2,c'_4,...]
\otimes\wedge(\hat{c}_1,\hat{c}_3,..),\end{equation} where
$c_\centerdot$ are the Chern polynomials and the accent and hat
indicate the place in the three factors of (\ref{Maslov}); the
homomorphism $\rho(\nabla^0,\nabla^1)$ is defined like on
$W(J_0,J_1)$, while using orthogonal and quasi-(skew)-metric
$A$-connections, respectively, and we get the same set of
characteristic classes. Then, by the same argument like for
\cite{V}, Theorem 4.4.37  we get
\begin{prop}\label{clsecFergen} The $A$-secondary characteristic
classes of $(V,g_\pm)$ are the real linear combinations of
cup-products of $A$-Pontrjagin classes of $V$ {\rm
\cite{F}} and classes of the form
\begin{equation}\label{clMV}
\mu_{2h-1}=[\Delta(\nabla^0,\nabla^1)c_{2h-1}]\in
H^{4h-3}(A).\end{equation}\end{prop}

The classes $\mu_{2h-1}$ will be called {\it simple $A$-secondary
characteristic classes}.
\begin{rem}\label{obsclFern} {\rm
If we start with an arbitrary vector bundle $(V,g_\pm)$, a
$K$-flat, quasi-(skew)-metric $A$-connection $\nabla^1$ may not
exist. Furthermore, if $\nabla^1$ exists, it may happen that all
the secondary characteristic classes vanish. For instance, if we
have a non-degenerate form $g_-$, a usual connection on the bundle
of $g_-$-canonical frames produces an $A$-connection $\nabla^1$
such that $\nabla^1 g_-=0$ and, since $K=0$, we get $A$-secondary
characteristic classes. Because of the $J_{\rm odd}$-homotopy of
orthogonal connections, these classes do not depend on the choice
of the orthogonal connection $\nabla^0$. Moreover, these classes
are independent of the skew-metric connection $\nabla^1$ because
of the existence of the link $(1-\tau)\nabla^1+\tau\nabla^{'1}$
between two such connections. But, the structure group of $V$ may
be reduced from the symplectic to the unitary group
\cite{V} and a unitary connection $\bar\nabla$ on $V$ will be skew-metric
and orthogonal simultaneously. From (\ref{defDelta}), and taking
$\nabla^0=\nabla^1=\bar\nabla$, we see that the secondary
characteristic classes above vanish.}\end{rem}
\section{Characteristic classes of morphisms}
Let $A$ be an arbitrary Lie algebroid on $M$, $V,W$ vector bundles
with the same basis $M$ and $\varphi:V\rightarrow W$ a morphism
over the identity on $M$. The $A$-connections $\nabla^V,\nabla^W$
on $V,W$, respectively, will be called $\varphi$-compatible if
$\nabla^W\circ\varphi=\varphi\circ\nabla^V$. An equivalent way to
characterize compatibility is obtained by considering the vector
bundle $S=V\oplus W^*$ , which is endowed with the $2$-forms
\begin{equation}\label{omegainS} g_\pm((v_1,\nu_1),(v_2,\nu_2))=
<\nu_2,\varphi(v_1)>\pm<\nu_1,\varphi(v_2)>,\end{equation}
$v_1,v_2\in V,\, \nu_1,\nu_2\in W^*$. It suffices to work with one
of these forms, but it is nice to mention that both may be used with
the same effect. The pair of $A$-connections $\nabla^V,\nabla^W$
produces an $A$-connection $\nabla^S=\nabla^V\oplus\nabla^{W^*}$ on
$S$, where $\nabla^{W^*}$ is defined by
$$<\nabla^{W^*}_a\nu,v>=(\sharp_Aa)<\nu,v>-<\nu,\nabla^W_av>,\hspace{3mm}\nu\in
W^*,v\in V.$$ A straightforward calculation shows that
$\nabla^V,\nabla^W$ are $\varphi$-compatible iff either
$\nabla^Sg_+=0$ or $\nabla^Sg_-=0$. We also notice that the forms
$g_\pm$ have the same annihilator
\begin{equation}\label{KinS} K=ker\,\varphi\times
ker\,^t\varphi\end{equation} where the index $t$ denotes
transposition.
\begin{prop}\label{FerinS}
If $V=A$, if $W=A'$ is a second Lie algebroid and if $\varphi$ is
a base-preserving Lie algebroid morphism, then there exist
$K$-flat, $\varphi$-compatible $A$-connections
$(\nabla,\nabla')$.\end{prop}
\begin{proof} We may proceed like in \cite{F}. Take a neighborhood
of $M$ where $\Gamma A,\Gamma A'$ have the fixed local bases
$(b_i),(b'_u)$. Define local $A$-connections
$\nabla^U,\nabla^{'U}$ by asking that
\begin{equation}\label{conexpeU} \nabla^U_{b_i}b_j=[b_i,b_j]_A,\;
\nabla^{'U}_{b_i}b'_u=[\varphi b_i,b'_u]_{A'},\end{equation}
then, extending the operators to arbitrary local cross sections in
accordance with the properties of a connection. Using the local
expression $\varphi b_i=\varphi_i^ub'_u$, it is easy to check that
$\varphi\circ\nabla^U=\nabla^{'U}\circ\varphi$. If we consider a
locally finite covering $\{U_\sigma\}$ of $M$ by such
neighborhoods $U$ and glue up the local connections by a
subordinated partition of unity $\{\theta_\sigma\in
C^\infty(M)\}$, we get $\varphi$-compatible, global
$A$-connections $\nabla,\nabla'$ defined by
\begin{equation}\label{lipireconex}\nabla_va(x)=\sum_{x\in
U_\sigma}\theta_\sigma(x)\nabla^{U_\sigma}_va(x),\;
\nabla'_va'(x)=\sum_{x\in
U_\sigma}\theta_\sigma(x)\nabla^{'U_\sigma}_va'(x),\end{equation}
where $x\in M,v\in A_x,a\in\Gamma A,a'\in\Gamma A'$.

Now, we notice that the local connections (\ref{conexpeU}) satisfy
the following properties
\begin{equation}\label{propcloc} \nabla^U_{b_i}a=[b_i,a]_A,\;
\nabla^{'U}_{b_i}a'=[\varphi b_i,a']_{A'}.\end{equation} Indeed, if
we put $a=f^jb_j,a'=h^ub'_u$, (\ref{conexpeU}) and the properties of
the Lie algebroid bracket imply (\ref{propcloc}). Furthermore, using
(\ref{propcloc}), it is easy to check the following properties of
the global compatible connections (\ref{lipireconex})
\begin{equation}\label{propluinabla} \nabla_vk(x)=[\tilde{v},k]_A(x),
\end{equation}
\begin{equation}\label{propluinabla'} <\nabla^{'*}_v\alpha'(x),a'(x)>
=(\sharp_Av)<\alpha',a'>-<\alpha'(x),[\varphi\tilde{v},a']_{A'}(x)>,\end{equation}
$\forall x\in M,k\in\Gamma(ker\,\varphi),a'\in\Gamma
A',\alpha'\in\Gamma(ker\,^t\varphi)$ and $\tilde{v}=\nu^ib_i$ is a
cross section of $\Gamma A$ that extends $v\in A_x$. The
restrictions put on $k,\alpha'$ ensure the correctness of the
passage from the covariant derivative to the Lie algebroid bracket
and the independence of the result on the choice of $\tilde{v}$.
Formulas (\ref{propluinabla}), (\ref{propluinabla'}) imply
$\varphi(\nabla_vk)=0$, $\nabla^{'*}_v\alpha'\circ\varphi=0$, which
means that $ker\,\varphi$ and $ker\,^t\varphi$ are preserved by the
connections $\nabla,\nabla'$, respectively.

Finally, if we denote $S=A\oplus A^{'*}$ and
$\nabla^S=\nabla\oplus\nabla^{'*}$, we can compute the curvature
$[R_{\nabla^S}(a_1,a_2)(k,\alpha')](x)$, which has components on $A$
and $A^{'*}$. The component on $A$ is
$$(\nabla_{a_1}\nabla_{a_2}-\nabla_{a_2}\nabla_{a_1}
-\nabla_{[a_1,a_2]_A})\tilde{k}(x)\stackrel{(\ref{propluinabla})}{=}
([\tilde{a}_1,[\tilde{a}_2,\tilde{k}]_A]_A$$ $$-
[\tilde{a}_2,[\tilde{a}_1,\tilde{k}]_A]_A-
[[\tilde{a}_1,\tilde{a}_2]_A,\tilde{k}]_A)(x)=0,$$ where tilde
denotes extensions to cross sections and the final result holds
because of the Jacobi identity. For the component on $A^{'*}$ we
get the following evaluation on any $a'\in\Gamma A'$:
$$<(\nabla^{'*}_{a_1}\nabla^{'*}_{a_2}-\nabla^{'*}_{a_2}\nabla^{'*}_{a_1}
-\nabla^{'*}_{[a_1,a_2]_A})\tilde{\alpha}',a'>(x)
\stackrel{(\ref{propluinabla'})}{=}<\tilde{\alpha}',
[\varphi\tilde{a}_2,[\varphi\tilde{a}_1,\tilde{k}]_{A'}]_{A'}$$ $$-
[\varphi\tilde{a}_1,[\varphi\tilde{a}_2,\tilde{k}]_{A'}]_{A'}-
[[\varphi\tilde{a}_1,\varphi\tilde{a}_2]_{A'},\tilde{k}]_{A'}>(x)=0,$$
where the annulation is justified by the Jacobi identity again.
Therefore, $$[R_{\nabla^S}(a_1,a_2)(k,\alpha')](x)=0,$$ which is the
meaning of $K$-flatness.
\end{proof}
\begin{rem}\label{obsrankconst} {\rm During the proof of Proposition \ref{FerinS} we
saw that $ker\,\varphi$ is preserved by $\nabla$, hence, it is
preserved by the parallel translation along the paths in the leaves
$L$ of $A$. This shows that $rank\,\varphi$ is constant along the
leaves $L$.}\end{rem}
\begin{rem}\label{nouindistins} {\rm If we use the
definition of $\nabla^{'*}$ in the left hand side of
(\ref{propluinabla'}) and take into account the relation
$ann\,ker\,^t\varphi=im\,\varphi$ we obtain the following equivalent
form of (\ref{propluinabla'}):
\begin{equation}\label{eqpropluinabla'}
\nabla'_va'(x)=[\varphi\tilde v,a']_{A'}(x)\;({\rm
mod.}\,im\,\varphi_x)\hspace{2mm}\forall x\in M,v\in A_x.
\end{equation}}\end{rem}
\begin{defin}\label{defdistins} {\rm
A pair of $\varphi$-compatible $A$-connections that satisfy the
properties (\ref{propluinabla}), (\ref{eqpropluinabla'}) will be
called a {\it distinguished pair} (in \cite{F} one uses the term
basic connections).}\end{defin}

Now, we see that we may use Proposition \ref{clsecFergen} in order
to get secondary characteristic classes for the bundle $S=A\oplus
A^{'*}$ endowed with the quasi-(skew)-metrics (\ref{omegainS}), with
a connection $\nabla^1=\nabla\oplus\nabla^{'*}$, where
$(\nabla,\nabla')$ is a distinguished pair of $A$-connections, and
with an orthogonal connection
$\nabla^0=\nabla^{g_A}\oplus\nabla^{g_{A'}*}$, where $g_A,g_{A'}$
are metrics on the bundles $A,A'$ and $\nabla^{g_A},\nabla^{g_{A'}}$
are corresponding orthogonal connections on $A,A'$.
\begin{defin}\label{defclmor} {\rm The above constructed
secondary characteristic classes of $A\oplus A^{'*}$ will be called
the {\it characteristic classes} of the base-preserving morphism
$\varphi$. In particular, one has the {\it simple characteristic
classes} $\mu_{2h-1}(\varphi)\in H^{4h-3}(A)$.}\end{defin}

The secondary characteristic classes of the Lie algebroid $A$
defined in \cite{F} are the simple characteristic classes of the
morphism $\varphi=\sharp_A:A\rightarrow TM$.
\begin{prop}\label{isomorfism} All the characteristic classes
of a base-preserving isomorphism $\varphi:A\rightarrow A'$ are
zero.\end{prop} \begin{proof} If $\varphi$ is an isomorphism, then
$g_-$ is non degenerate and we are in the situation discussed in
Remark \ref{obsclFern}.\end{proof}

Thus, the characteristic classes of a morphism may be seen as a
measure of its non-isomorphic character.
\begin{prop}\label{invarlaconex} The characteristic
classes of a base preserving morphism $\varphi:A\rightarrow A'$ of
Lie algebroids do not depend on the choice of the orthogonal
connection and of the distinguished pair of compatible connections
required by their definition.\end{prop}
\begin{proof} The proposition is a consequence of Corollary
\ref{corolhomotopie}. In the previous section we have seen that two
orthogonal $A$-connections are $J_{\rm odd}$-homotopic. On the other
hand, take two $\varphi$-distinguished pairs of $A$-connections
$\nabla,\nabla';\tilde{\nabla},\tilde{\nabla}'$. Then, it is easy to
check that, $\forall t\in[0,1]$,
$(1-t)\nabla+t\tilde{\nabla},(1-t)\nabla'+t\tilde{\nabla}'$ is a
$\varphi$-distinguished pair again. Therefore, $J_{\rm
odd}$-homotopy also holds for the corresponding quasi-(skew)-metric
connections on $S$ and we are done.\end{proof}

We also have another consequence of Corollary \ref{corolhomotopie}:
\begin{prop}\label{morfismehomotope} Two homotopic, base-preserving
morphisms $\varphi_0,\varphi_1:A\rightarrow A'$ of Lie algebroids
have the same secondary characteristic classes.\end{prop}
\begin{proof} By homotopic morphisms we understand  morphisms
$\varphi_0,\varphi_1$ that are linked by a differentiable family
of morphisms $\varphi_\tau:A\rightarrow A'$ $(0\leq\tau\leq1)$.
The corresponding forms $g_{+,\tau}$ on $S=A\oplus A^{'*}$ are
different, but, still, all the connections $\nabla^{1,\tau}$
required in the construction of the secondary classes have
skew-symmetric local connection and curvature matrices. Therefore,
the $J_{\rm odd}$-homotopy holds and we are done.\end{proof}
\begin{rem}\label{clasebicaract} {\rm
In the case of an arbitrary pair of morphisms
$\varphi_0,\varphi_1:A\rightarrow A'$ we can measure the
difference between the secondary characteristic classes as
follows. Notice the existence of the {\it bi-characteristic
classes}
$\bar{\mu}_{2h-1}(\varphi_1,\varphi_2)=[\Delta(\nabla^1,\nabla^2)c_{2h-1}]\in
H^{4h-3}(A)$ where $\nabla^1,\nabla^2$ are $A$-connections defined
on $S=A\oplus A^{'*}$ by distinguished, $\varphi_{1,2}$-compatible
connections respectively. Then, formula (\ref{diffBott}) yields
$$d_A\Delta(\nabla^0,\nabla^1,\nabla^2)c_{2h-1}=\Delta(\nabla^0,\nabla^1)c_{2h-1}
+\Delta(\nabla^1,\nabla^2)c_{2h-1}+\Delta(\nabla^2,\nabla^0)c_{2h-1},$$
where $\nabla^0$ is an orthogonal connection on $S$. Accordingly,
we get
\begin{equation}\label{claseperechi} {\mu}_{2h-1}(\varphi_1)-
{\mu}_{2h-1}(\varphi_2)=
\bar{\mu}_{2h-1}(\varphi_1,\varphi_2).\end{equation}}\end{rem}

In what follows we give explicit local expressions of $A$-forms
that represent the characteristic classes $\mu_{2h-1}(\varphi)$.
Take a point $x\in M$ and an open neighborhood $U$ of $x$
diffeomorphic to a ball. Assume that $(\nabla^U,\nabla^{'U})$ and
$(\nabla,\nabla')$ are pairs of local, respectively global,
distinguished, $\varphi$-compatible $A$-connections on $A,A'$.
Then, if $0\leq\chi\in C^\infty(M)$ is equal to $1$ on the compact
closure $\bar V$ of the open neighborhood $V\subseteq U$ of $x$
and equal to $0$ on $M\backslash U$, then the convex combinations
$$\bar{\nabla}=\chi\nabla^U+(1-\chi)\nabla,\,
\bar{\nabla}'=\chi\nabla^{'U}+(1-\chi)\nabla'$$
define a global pair of distinguished $A$-connections that
coincides with $(\nabla^U,\nabla^{'U})$ on $V$.

Accordingly, in formula (\ref{clMV}) for $S=A\oplus A^{'*}$ we may
always use a connection $\nabla^1$ such that the expressions
(\ref{conexpeU}) hold on the neighborhood $V$. Then, if we denote
\begin{equation}\label{expresiiloc1} \begin{array}{c}
[b_i,b_j]_A=\gamma_{ij}^kb_k,\,
[b'_u,b'_v]_{A'}=\gamma_{uv}^{'w}b'_w,\vspace{2mm}\\
\sharp_Ab_i=\rho_i^j\frac{\partial}{\partial x^j},\,
\sharp_{A'}b'_u=\rho_u^{'j}\frac{\partial}{\partial x^j}
\,\varphi(b_i)=\varphi_i^sb'_s\end{array}\end{equation}
(remember that we use the Einstein summation convention), where
$x^i$ are local coordinates on $M$ and $(b_i)(,b'_u)$ are the
bases used in (\ref{conexpeU}), we get the following connection
matrix of $\nabla^1$ on the neighborhood $V$
\begin{equation}\label{expresiiloc2} \left( \begin{array}{cc}
\gamma_{ij}^kb^{*i}&0\vspace{2mm}\\
0&(-\varphi_i^t\gamma_{tu}^{'s}+\rho^{'j}_u
\frac{\partial\varphi_i^s}{\partial x^j})b^{*i}\end{array}\right)\end{equation}
(in (\ref{expresiiloc2}), $b^{*i}$ is the dual basis of $b_i$).

Furthermore, let $g^U,g^{'U}$ be local metrics on $A,A'$ such that
$(b_i),(b'_u)$ are orthonormal bases and $g,g'$ arbitrary, global
metrics on $A,A'$. Then, define the metrics
$$\chi g^U+(1-\chi)g,\,\chi g^{'U}+(1-\chi)g'$$ and take
an orthogonal connection $\nabla^0$ whose components are
corresponding orthogonal connections. The connection matrix of
$\nabla^0$ on the neighborhood $V$, with respect to the same local
bases like in (\ref{expresiiloc2}), will be of the form
\begin{equation}\label{expresiiloc3} \left(
\begin{array}{cc}\varpi_i^j&0\vspace{2mm}\\ 0&-\varpi_s^{'t}
\end{array}\right),\end{equation} where
$(\varpi_i^j),(\varpi_s^{'t})$ are skew-symmetric matrices of
local $1$-$A$-forms.

If these connections $\nabla^0,\nabla^1$ are used, then, along
$V$, the difference matrix $\alpha$ of formula (\ref{defDelta}) is
the difference between the matrices (\ref{expresiiloc2}) and
(\ref{expresiiloc3}). Furthermore, we can compute the matrix
$\Omega_{(\tau)}$ by using formula (\ref{Omegat}), where
$\Omega_{(0)}$ is a skew-symmetric matrix. The final result may be
formulated as follows
\begin{prop}\label{propexpresmu} If a point $x\in M$ is fixed,
there exist global representative $A$-forms $\Xi_{2h-1}\in
\Omega^{4h-3}(A)$ of the characteristic classes $\mu_{2h-1}$ such
that \begin{equation}\label{localmu} \Xi_{2h-1}|_V=
\frac{1}{(2h-2)!}\int_0^1\left(
\delta^{\sigma_1...\sigma_{2h-1}}_{\kappa_1...\kappa_{2h-1}}
\alpha_{\sigma_1}^{\kappa_1}\wedge\Omega_{(\tau),\sigma_2}^{\kappa_2}
\wedge...\wedge\Omega_{(\tau),\sigma_{2h-1}}^{\kappa_{2h-1}}\right)d\tau,
\end{equation} for some neighborhood $V$ of $x$. In
(\ref{localmu}), the factors are the entries of the matrices
$\alpha,\Omega_{(\tau)}$ given by formulas (\ref{expresiiloc2}),
(\ref{expresiiloc3}) and Greek indices run from $1$ to
$dim\,A+dim\,A'$.\end{prop}
\begin{proof} Use the expression (\ref{Chern}) of the Chern
polynomials and the connections $\nabla^0,\nabla^1$ constructed
above. \end{proof}

The difficulty in using Proposition \ref{propexpresmu}, besides
its complexity in the case $h>1$, consists in the fact that
formula (\ref{localmu}) does not define global $A$-forms; for
neighborhoods of different points $x_1\neq x_2$ we have different
pairs of distinguished connections $\bar{\nabla},\bar{\nabla}'$.
However, we can use Proposition \ref{propexpresmu} in order to
extend a result proven for a Lie algebroid $A$
($\varphi=\sharp_A$) in
\cite{F}:
\begin{prop}\label{comparmodul}
The secondary class $\mu_1(\varphi)$ is equal to the modular class
of the morphism $\varphi$.\end{prop}
\begin{proof} Recall that the modular
class of a morphism is defined by
$\mu(\varphi)=\mu(A)-\varphi^*\mu(A')\in H^1(A)$, where
$\mu(A),\mu(A')$ are the modular classes of the Lie algebroids
$A,A'$, respectively, \cite{{GMM},{KW},{KGW}}. Furthermore, the
modular class $\mu(A)$ is defined as follows
\cite{{ELW},{F},{H},{KGW}}. The line bundle
$\wedge^sA\otimes\wedge^mT^*M$ $(s=rank\,A)$ has a flat
$A$-connection defined, by means of local bases, as follows
\begin{equation}\label{flattop}
\nabla_{b_i}((\wedge_{j=1}^sb_j)\otimes(\wedge_{h=1}^mdx^h))
=\sum_{j=1}^sb_1\wedge...\wedge [b_i,b_j]_A\wedge...\wedge
b_s\otimes(\wedge_{h=1}^mdx^h)\end{equation}
$$+(\wedge_{j=1}^sb_j)\otimes
L_{\sharp_Ab_i}(\wedge_{h=1}^mdx^h),$$ where $L$ is the Lie
derivative. Then, for
$\sigma\in\Gamma(\wedge^sA\otimes\wedge^mT^*M)$ (which exists if
the line bundle is trivial; otherwise we go to its double
covering), one has $\nabla_a\sigma=\lambda(a)\sigma$ where
$\lambda$ is a $d_A$-closed $1$-$A$-form and defines the
cohomology class $\mu(A)$, which is independent on the choice of
$\sigma$.

From (\ref{flattop}) it follows easily that $\mu(A),\mu(A')$ are
represented by the $A$-forms
\begin{equation}\label{formemodulare}
\lambda=\sum_{i,k,j}(\gamma_{ik}^k+\frac{\partial\rho_i^j}{\partial
x^j})b^{*i},\,
\lambda'=\sum_{s,t,h}(\gamma_{st}^t+\frac{\partial\rho_s^{'j}}{\partial
x^j})b^{'*s}\end{equation} where the notation is that of
(\ref{expresiiloc1}). Notice that, even though the expressions
(\ref{formemodulare}) are local, the forms $\lambda,\lambda'$ are
global $A$-forms because the connection that was used in their
definition is global.

On the other hand, using formulas (\ref{expresiiloc2}),
(\ref{expresiiloc3}) and since the trace of a skew-symmetric
matrix is zero, we may see that the $A$-form $\Xi_1$ defined in
Proposition \ref{propexpresmu} is such that
$\Xi|_{1V}=(\lambda-\varphi^*\lambda')|_{V}$, where $V$ is a
neighborhood of a fixed point $x\in M$. Accordingly, there exists
a locally finite, open covering $\{V_\alpha\}$ of $M$ and there
exists a family of pairs of $A$-connections
$(\nabla^{0\alpha},\nabla^{1\alpha})$ that provide representative
$1$-$A$-forms $\Xi_{1\alpha}$ of the characteristic class
$\mu_1(\varphi)$ such that
\begin{equation}\label{auxmodular}
\Xi_{1\alpha}|_{V_\alpha}=(\lambda-\varphi^*\lambda')|_{V_\alpha}.\end{equation}
Then, if we take a partition of unity $\{\theta_\alpha\in
C^\infty(M)\}$ subordinated to $\{V_\alpha\}$ and glue up the
families $\nabla^{0\alpha},\nabla^{1\alpha}$, like in
(\ref{lipireconex}), we get connections $\nabla^0,\nabla^1$ that
define the representative $A$-form
$$\Xi_1(x)=\sum_{x\in
V_\alpha}\theta_\alpha(x)\Xi_{1\alpha}(x)=(\lambda-\varphi^*\lambda')(x),
\hspace{2mm}x\in M$$ of $\mu_1(\varphi)$.
This justifies the required conclusion.\end{proof}
\begin{example}\label{Ps-Nij} {\rm  An
interesting example appears on a Poisson-Nijenhuis manifold
$(M,P,N)$, where $P$ is a Poisson bivector field and $N$ is a
Nijenhuis tensor. Then $^tN:(T^*M,N\circ\sharp_P)
\rightarrow (T^*M,\sharp_P)$ is a morphism of cotangent Lie
algebroids. The modular class of the morphism $^tN$ was studied in
\cite{DF} and it would be interesting to get information about
other characteristic classes of this morphism.}\end{example}

The calculation of the classes $\mu_{2h-1}$ for $h>1$ is much more
complicated. One of the difficulties is the absence of a global
construction of a distinguished pair of connections.
\begin{example}\label{exdistins} {\rm Let $\varphi:A\rightarrow A$ be
an endomorphism of the Lie algebroid $A$ and assume that there
exists an $A$-connection $\nabla$ on $A$ that satisfies condition
(\ref{propluinabla}) and whose torsion
$$T_\nabla(a_1,a_2)=\nabla_{a_1}a_2-\nabla_{a_2}a_1-[a_1,a_2]_A,\hspace{2mm}
a_1,a_2\in\Gamma A,$$ takes values in $K=ker\,\varphi$. Then, it
is easy to check that the formula
$$\nabla'_{a_1}a_2=[\varphi a_1,a_2]_A + \varphi\nabla_{a_2}a_1$$
defines a second $A$-connection that is $\varphi$-compatible with
$\nabla$ and satisfies condition (\ref{eqpropluinabla'}).
Therefore, $(\nabla,\nabla')$ is a distinguished
pair.}\end{example}

Another difficulty is produced by the complicated character of the
expression (\ref{localmu}). A simple example follows.
\begin{example}\label{exhmare} {\rm If the Lie algebroids
$A,A'$ have anchors zero, the $A$-connections are tensors and
formula (\ref{expresiiloc2}) gives the local connection matrices
of a global, flat $A$-connection $\nabla^1$ as required in the
definition of the characteristic classes (flatness is just Jacobi
identity). In the simplest case
$A=M\times\mathcal{G},A'=M\times\mathcal{G}'$ where
$\mathcal{G},\mathcal{G}'$ are Lie algebras, we may take
$\varpi_i^j=0$ in (\ref{expresiiloc3}), which gives a flat metric
connection $\nabla^0$. Then, formula (\ref{Omegat}) reduces to
$$\Omega_{(\tau)}=\tau(1-\tau)\alpha\wedge\alpha$$ where $\alpha$
is the matrix (\ref{expresiiloc2}). Accordingly, like in \cite{V},
Theorem 4.5.11, we get the representative $A$-forms
$$\Xi_{2h-1}=\frac{1}{(2h-2)!}\nu_h
\delta^{\sigma_1...\sigma_{2h-1}}_{\kappa_1...\kappa_{2h-1}}
\alpha_{\sigma_1}^{\kappa_1}\wedge\alpha_{\sigma_2}^{\lambda_2}
\wedge\alpha_{\lambda_2}^{\kappa_2}\wedge...\wedge
\alpha_{\sigma_2h-1}^{\lambda_{2h-1}}
\wedge\alpha_{\lambda_{2h-1}}^{\kappa_{2h-1}}$$ of the classes
$\mu_{2h-1}$, where $\alpha^{\cdot}_{.}$ are the entries of the
matrix (\ref{expresiiloc2}) and
$$\nu_h=\int_0^1\tau(1-\tau)d\tau=\sum_{i=1}^{2h-2}(-1)^{h+i+1}\frac{2^i}{4h-i-3}
\left(\begin{array}{c}2h-2\vspace{2mm}\\ i\end{array}\right).$$
}\end{example}
\begin{rem}\label{bazediferite} {\rm So far, we do not have
a good definition of characteristic classes of a morphism between
Lie algebroids over different bases. Using the terminology and
notation of
\cite{KGW}, let us
consider a morphism \begin{equation}\label{genmorf}
\begin{array}{ccc} A&
\stackrel{\varphi}{\rightarrow}&B\vspace{2mm}\\ \downarrow&
&\downarrow\vspace{2mm}\\
M&\stackrel{f}{\rightarrow}&N\end{array}\end{equation} between the
Lie algebroids $A,B$ and assume that the mapping $f$ is
transversal to the Lie algebroid $B$. Then, Proposition 3.11 of
\cite{KGW} tells us that $\varphi=f_B^{!!}\circ\varphi'$, where
$f_B^{!!}:f^{!!}B\rightarrow B$, $\varphi':A\rightarrow f^{!!}B$
are the natural projections of the pullback Lie algebroid
$f^{!!}B$. Furthermore, Proposition 3.12 of \cite{KGW}  tells that
the modular class of the non base preserving morphism $\varphi$ is
equal to the modular class of the base preserving morphism
$\varphi'$. This equality may be extended by definition to all the
characteristic classes of $\varphi$, but it is not clear whether
this definition is good (it does not loose information about
$\varphi$) even in the indicated particular case.}\end{rem}
\section{Relative characteristic classes}
From Proposition \ref{comparmodul} and a known result on modular
classes (\cite{KGW}, formula (2.5)) we see that the first class
$\mu_1(\varphi)$ has a nice behavior with respect to the
composition of morphisms namely, for the morphisms
$\varphi:A\rightarrow A',\psi:A'\rightarrow A''$ one has
\begin{equation}\label{modcompus} \mu_1(\psi\circ\varphi)=
\mu_1(\varphi)+\varphi^*(\mu_1(\psi)).\end{equation} In this section
we give a proof of (\ref{modcompus}) by means of the definition of
the characteristic classes of a morphism and we shall see why the
result does not extend to the higher classes $\mu_{2h-1}$, $h>1$.
The proof will use a kind of relative characteristic classes that
are interesting in their own right; in particular, we will show
that the relative classes defined by the jet Lie algebroid $J^1A$
\cite{CF} are cohomological images of the absolute characteristic
classes of a morphism $\varphi:A\rightarrow A'$.

Like in the definition of the characteristic classes of $\varphi$
we can produce characteristic classes of $\psi:A'\rightarrow A''$
modulo $\varphi:A\rightarrow A'$ as follows. Take the Lehmann
morphism $\rho^*(D^0,D^1)$ for an orthogonal $A$-connection $D^0$
on the vector bundle $A'\oplus A^{''*}$ associated with a sum of
Euclidean metrics $g_{A'},g_{A''}$ and an $A$-connection $D^1$ on
$A'\oplus A^{''*}$, which is the sum of {\it distinguished
$A$-connections} $\nabla',\nabla''$ on $A',A''$, respectively.
Here by a distinguished pair we mean a pair of $A$-connections
$(\nabla',\nabla'')$ that satisfies the following properties
\begin{equation}\label{distins2} \begin{array}{l}
\psi\nabla'_aa'=\nabla''_a(\psi a'),\;\;a\in A_x\,(x\in M),\,
a'\in\Gamma A',\vspace{2mm}\\
\nabla'_ak(x)=[\varphi\tilde{a},k]_{A'}(x),\;\;k\in\Gamma ker\,\psi,
\vspace{2mm}\\
\nabla''_aa''(x)=[\psi\varphi\tilde{a},a'']_{A''}(x)\;\;({\rm
mod.\,}im\,\psi),\end{array}\end{equation} where the sign tilde
denotes the extension to a cross section. One can construct a
$\psi$-distinguished pair of $A$-connections $\nabla',\nabla''$ by
replacing the local formulas (\ref{conexpeU}) by
\begin{equation}\label{conexpeU1} \nabla^{'U}_{b_i}b'_{j'}
=[\varphi b_i,b'_{j'}]_{A'},\;
\nabla^{''U}_{b_i}b''_{j''}=[\psi\varphi b_i,b''_{j''}]_{A''},
\end{equation}
then gluing the local connections via a partition of unity. (In
(\ref{conexpeU1}) $(b_i),(b'_{i'}),(b''_{i''})$ are local bases of
$\Gamma A,\Gamma A',\Gamma A''$, respectively.)
\begin{defin}\label{defclrel} {\rm The
characteristic $A$-cohomology classes in $im\,\rho^*(D^0,D^1)$
will be called {\it relative characteristic classes} of $\psi$
modulo $\varphi$. In particular, $$\mu_{2h-1}(\psi\,{\rm
mod.}\,\varphi) =[\Delta(D^0,D^1)]\in H^{4h-3}(A)$$ are the {\it
simple relative characteristic classes}.}\end{defin}
\begin{prop}\label{1relativ} For $h=1$, the relative and absolute
characteristic class $\mu_1$ of the morphism $\psi$ are related by
the equality \begin{equation}\label{pullbackmu1}
\mu_{1}(\psi\,{\rm mod.}\,\varphi)
=\varphi^*\mu_{1}(\psi).\end{equation}
\end{prop}
\begin{proof} By absolute classes we understand characteristic
classes $\mu_{2h-1}(\psi) \in H^{4h-3}(A')$. The partition of
unity argument given for (\ref{expresiiloc2}) shows that we may
assume the following local expressions of distinguished
$A'$-connections on $A',A''$
\begin{equation}\label{conexpeU2} \bar{\nabla}^{'U}_{b'_{i'}}b'_{j'}
=[b'_{i'},b'_{j'}]_{A'},\;
\bar{\nabla}^{''U}_{b'_{i'}}b''_{j''}=[\psi b'_{i'},b''_{j''}]_{A''}.
\end{equation}
Connections (\ref{conexpeU2}) induce $A$-connections
$\tilde{\nabla}',\tilde{\nabla}''$ and we shall compute the local
matrices of the induced connections. By definition, we have
$$\tilde{\nabla}^{'U}_{b_i}b'_{j'}=
\bar{\nabla}^{'U}_{\varphi b_i}b'_{j'},\;
\tilde{\nabla}^{''U}_{b_i}b''_{j''}=
\bar{\nabla}^{''U}_{\varphi b_i}b''_{j''}$$ and it is easy to
check that the $A$-connections
$\tilde{\nabla}^{'U},\tilde{\nabla}^{''U}$ satisfy conditions
(\ref{distins2}). Therefore,
$\tilde{\nabla}^{'U},\tilde{\nabla}^{''U}$ may be used in the
calculation of the relative characteristic classes of $\psi$ mod.
$\varphi$. If we denote $\varphi b_i=\varphi_i^{j'}b'_{j'}$ and
use expressions (\ref{conexpeU2}) and the properties of the Lie
algebroid brackets we obtain the local connection matrices
\begin{equation}\label{conexinduse} \begin{array}{l}
\tilde{\omega}^{'k'}_{j'}=
\varphi^*\bar{\omega}^{'k'}_{j'}-<d_{A'}\varphi_i^{k'},b_{j'}>b^{*i},\vspace{2mm}\\
\tilde{\omega}^{''k''}_{j''}=
\varphi^*\bar{\omega}^{''k''}_{j''}-<d_{A''}\varphi_i^{k'},\psi b_{j'}>b^{*i}.
\end{array}\end{equation}
Formula (\ref{conexinduse}) allows us to write down the local
connection matrix of the connection
$D^1=\tilde{\nabla}'+\tilde{\nabla}^{''*}$ required by the
definition of the relative classes. Furthermore, we may assume
that the local matrix of the orthogonal connection $D^0$ that we
use is skew-symmetric. Accordingly, and since $\psi$ is a Lie
algebroid morphism, (\ref{conexinduse}) yields
$$
\Delta(D^0,D^1)c_1=tr\,\left(\begin{array}{cc}
\tilde{\omega}^{'k'}_{j'}&0\vspace{2mm}\\0&
-\tilde{\omega}^{''j''}_{k''}\end{array}\right) =\varphi^*
tr\,\left(\begin{array}{cc}
\bar{\omega}^{'k'}_{j'}&0\vspace{2mm}\\0&
-\bar{\omega}^{''j''}_{k''}\end{array}\right)=
\varphi^*\Delta(\bar{\nabla}^0,\bar{\nabla}^1)c_1,$$
where $\bar{\nabla}^1=\bar{\nabla}'+\bar{\nabla}^{''*}$ and
$\bar{\nabla}^0$ is an orthogonal $A'$-connection on $A'\oplus
A^{''*}$. This result justifies (\ref{pullbackmu1}).\end{proof}
\begin{prop}\label{2relativ}
For $h=1$, the relative and absolute characteristic class $\mu_1$
of the morphisms $\varphi,\psi$ are related by the equality
\begin{equation}\label{eqKGW} \mu_1(\psi\circ\varphi)=
\mu_1(\varphi)+\mu_1(\psi\,{\rm mod.}\,\varphi).\end{equation}\end{prop}
\begin{proof} In
the computation of $\mu_1(\psi\circ\varphi)$ we may use an
$A$-connection $\nabla+\nabla^{''*}$ on $A\oplus A^{''*}$ where,
on the specified neighborhood $U$, $\nabla$ is given by
(\ref{conexpeU}) and $\nabla''$ is given by (\ref{conexpeU1}),
while in the computation of $\mu_{1}(\psi\,{\rm mod.}\,\varphi)$
we shall use the connections $\nabla',\nabla''$ of
(\ref{conexpeU1}). Thus, the non-zero blocks of the local
difference matrix $\alpha$ that enters into the expression of the
representative $1-A$-form of $\mu_1(\psi\circ\varphi)$ are given
by the local matrix of
\begin{equation}\label{blocuri}
\nabla''-\nabla=\nabla''-\nabla'+\nabla'-\nabla\end{equation} and
the opposite of its transposed matrix (in spite of the notation,
calculation (\ref{blocuri}) is for the connection matrices not for
the connections). Then, if we use orthogonal connections of
metrics where the bases used in (\ref{conexpeU1}) are orthonormal
bases (therefore, with trace zero), formula (\ref{blocuri})
justifies (\ref{eqKGW}).\end{proof}
\begin{corol}\label{corolarKGW} The characteristic class $\mu_1$
of a composed morphism $\psi\circ\varphi,\psi$ is given by formula
(\ref{modcompus}).\end{corol} \begin{proof} The result is an
obvious consequence of formulas (\ref{pullbackmu1}) and
(\ref{eqKGW}).\end{proof}
\begin{rem}\label{obsfinala} {\rm Formulas (\ref{modcompus}),
(\ref{pullbackmu1}), do not hold for $h>1$ because of the more
complicated expression of the polynomials $c_{2h-1}$ (there is no
nice formula for the determinant of a sum of matrices).}\end{rem}

We finish by showing the relation between the characteristic
classes of the base-preserving Lie algebroid morphism
$\varphi:A\rightarrow A'$ and the relative classes defined by the
first jet Lie algebroid $J^1A$; for $\varphi=\sharp_A:A\rightarrow
TM$ this relation was established in \cite{CF}.

The first jet bundle $J^1A$ may be defined as follows. Let $D$ be
a $TM$-connection on $A$ and let $Da$ denote the covariant
differential of a cross section $a\in\Gamma A$ (i.e.,
$Da(X)=D_Xa$, $X\in\Gamma TM$). The properties of a connection
tell us that $Da\in Hom(TM,A)$ and, if $a(x_0)=0$ for some point
$x_0\in M$, then $Da(x_0):T_{x_0}M\rightarrow A_{x_0}$ is a linear
mapping that is independent of the choice of the connection $D$.
(This is not true if $a(x_0)\neq0$.) If $(x^h)$ are local
coordinates of $M$ around $x_0$ and $(b_i)$ is a local basis of
$\Gamma A$, and if $a=\xi^i(x^h)b_i$, the local matrix of
$Da(x_0)$ is $(D_h\xi^i(x_0))$ (the covariant derivative tensor),
which reduces to $(\partial\xi^i/\partial x^h(x_0))$ if
$\xi^i(x_0)=0$.

Now, for any point $x_0\in M$, the space of $1$-jets of cross
sections of $A$ at $x_0$ is
\begin{equation}\label{spjetx}
J^1_{x_0}A=\Gamma A/\{a\in\Gamma
A\,/\,a(x_0)=0,Da(x_0)=0\}\end{equation} and each $a\in\Gamma A$
defines an element $j_x^1a\in J^1_{x_0}A$ called the $1$-jet of
$a$ at $x_0$. With the local coordinates and basis considered
above, we may write
$$a=\xi^i(x^h)b_i=(\xi^i(x_0)+\frac{\partial\xi^i}{\partial
x^h}(x_0)(x^h-x^h(x_0))+o((x^h-x^h(x_0))^2))b_i.$$ Hence,
$$j_{x_0}^1a=\xi^i(x_0)j_{x_0}^1b_i+
\frac{\partial\xi^i}{\partial x^h}(x_0)j_{x_0}^1((x^h-x^h(x_0))b_i)$$ and
\begin{equation}\label{bazejet}
j_{x_0}^1b_i,\,j_{x_0}^1((x^h-x^h(x_0))b_i)=j_{x_0}^1
(x^hb_i)-x^h(x_0)j_{x_0}^1b_i\end{equation} is a basis of the
vector space $J^1_{x_0}A$ such that
$(\xi^i(x_0),\partial\xi^i/\partial x^h(x)(x_0))$ are coordinates
with respect to this basis.

A change of the local coordinates and basis of $A$ gives the
transition formulas
\begin{equation}\label{schimbcoordjet}
\tilde{x}^h=\tilde{x}^h(x^k),\,\tilde{\xi}^i=\lambda^i_j(x^k)\xi^j,\,
\frac{\partial\tilde{\xi}^i}{\partial\tilde{x}^h}=
\frac{\partial{x}^k}{\partial\tilde{x}^h}\left(
\frac{\partial\lambda_j^i}{\partial{x}^k}\xi^j+\lambda_j^i
\frac{\partial{\xi}^j}{\partial{x}^k}\right)\end{equation}
and may be seen as the composition of the change of the
coordinates $(x^h)$ with the change of the basis $(b_i)$, while
the order of the two changes is irrelevant. This remark allows for
an easy verification of the fact that the change of the
coordinates discovered above in $J^1_{x_0}A$ has the cocycle
property. Accordingly, (\ref{schimbcoordjet}) shows that
$J^1A=\cup_{x\in M}J^1_xA$ has a natural structure of a
differentiable manifold and vector bundle $\pi:J^1A\rightarrow M$
over $M$ called the first jet bundle of $A$.

From (\ref{bazejet}), we see that $(j^1b_i,j^1(x^hb_i))$ is a
local basis of cross sections of $J^1A$ at each point of the
coordinate neighborhood where $x^h$ are defined. This basis
consists of $1$-jets of local cross sections of $A$, therefore,
the cross sections of $J^1A$ are locally spanned by $1$-jets of
cross sections of $A$ over $C^\infty(M)$. In the case of a Lie
algebroid $A$, the previous remark allows us to define a Lie
algebroid structure on $J^1A$ by putting
\begin{equation}\label{strluiJ1A} \sharp_{J^1A}(j^1a)=\sharp_Aa,\;
[j^1a_1,j^1a_2]_{J^1A}=j^1[a_1,a_2]_A\end{equation} and by
extending the bracket to general cross sections via the axioms of
a Lie algebroid. We refer the reader to Crainic and Fernandes
\cite{CF} for details. A general, global expression of the Lie
algebroid bracket of $J^1A$ was given by Blaom \cite{B}.

Moreover, (\ref{strluiJ1A}) shows that the natural projection
$\pi^1:J^1A\rightarrow A$, $\pi^1(j^1a)=a$ is a base-preserving
morphism of Lie algebroids and, if $\varphi:A\rightarrow A'$ is a
morphism of Lie algebroids, we may define relative characteristic
classes of $\varphi$ modulo $\pi^1$. Following \cite{CF}, there
exist flat $J^1A$-connections $\nabla^{j^1},\nabla^{'j^1}$ on
$A,A'$, respectively, given by
\begin{equation}\label{Jconex}
\nabla^{j^1}_{fj^1a_1}a_2=f[a_1,a_2]_A,\, \nabla^{'j^1}_{fj^1a}a'=f[\varphi
a,a']_{A'},\end{equation} where $a,a_1,a_2\in\Gamma A, a'\in\Gamma
A',f\in C^\infty(M)$. These connections obviously satisfy
conditions (\ref{distins2}), hence,
$D^1=\nabla^{j^1}+\nabla^{'*j^1}$ is a $J^1A$-connection on
$A\oplus A^{'*}$ that may be used in Definition \ref{defclrel} for
the present case. We shall prove the following result
\begin{prop}\label{propCF} The relative characteristic classes of
$\varphi$ modulo $\pi^1$ are the images of the corresponding
absolute characteristic classes of $\varphi$ by the homomorphism
$\pi^{1*}:H^*(A)\rightarrow H^{*}(J^1A)$.\end{prop}
\begin{proof} Here, we have the particular case of the situation
that existed in Proposition \ref{1relativ} where $\pi^1$ comes
instead of $\varphi$ and $\varphi$ comes instead of $\psi$.
Therefore, we may construct connections that correspond to
(\ref{conexpeU2}) and the induced $J^1A$-connections and get the
corresponding formulas (\ref{conexinduse}). If we use the local
bases (\ref{bazejet}) of $\Gamma J^1A$, the components
$\varphi_i^{k'}$ that appear in (\ref{conexinduse}) are constant
and (\ref{conexinduse}) simply tell us that the local connection
forms of the induced connections are the pullback of the
connection forms of the connections (\ref{conexpeU2}) by $\pi^1$.
Of course, the same holds for the curvature forms, and, if we also
use a $J^1A$-orthogonal connection of $A\oplus A^{'*}$ that is
induced by an $A$-orthogonal connection, we see that $\pi^{1*}$
commutes with the Lehmann morphism, which implies the required
result.\end{proof}
\hspace*{7.5cm}{\small \begin{tabular}{l} Department of
Mathematics\\ University of Haifa, Israel\\ E-mail:
vaisman@math.haifa.ac.il \end{tabular}}
\end{document}